\documentclass{article}
\usepackage{amsfonts}
\usepackage{amsmath}
\usepackage{amssymb}
\usepackage{amsthm}
\usepackage{fullpage}
\usepackage{cleveref}
\usepackage{mathtools}
\usepackage{url}

\title{An elementary proof of the Koml\'os conjecture}
\author{
Sankeerth Rao Karingula\\
{Agentin AI}\\
{\texttt{sankeerthrao@agentin.ai}}
\and
Shachar Lovett\thanks{Supported by Simons Investigator Award \#929894 and NSF award CCF-2425349.}\\
{University of California, San Diego}\\
{\texttt{slovett@ucsd.edu}}
}

\newtheorem{theorem}{Theorem}[section]
\newtheorem{lemma}[theorem]{Lemma}
\newtheorem{claim}[theorem]{Claim}
\newtheorem{definition}[theorem]{Definition}

\newtheorem{conjecture}[theorem]{Conjecture}

\crefname{claim}{claim}{claims}
\Crefname{claim}{Claim}{Claims}

\newcommand{\R}{\mathbb{R}}
\newcommand{\Q}{\mathbb{Q}}

\newcommand{\eps}{\varepsilon}
\newcommand{\SD}[2]{d_{\mathrm{TV}}(#1,#2)}
\newcommand{\SDS}[2]{\Delta(#1,#2)}
\newcommand{\supp}{\operatorname{supp}}

\newcommand{\norm}[1]{\left \| #1 \right \|}

\begin{document}

\maketitle

\begin{abstract}
We give an elementary proof of the Koml\'os conjecture by simplifying
the recent proof of Guo, Fang, and Lu.
We show that any vectors $v_1,\ldots,v_n\in\mathbb{R}^d$ with
$\|v_i\|_2\le1$ admit signs $\varepsilon_i\in\{-1,1\}$ such that
$\|\sum_{i=1}^n\varepsilon_i v_i\|_\infty\le36$.
The proof uses only elementary combinatorial and probabilistic
arguments and basic calculus.
\end{abstract}

\section{Introduction}

The Koml\'os conjecture is a longstanding problem in discrepancy
theory, recorded by Spencer~\cite{spencer1987ten}.
It generalizes the $O(\sqrt{n})$ discrepancy bound of
Spencer~\cite{spencer1985six} and contains the conjecture of
Beck and Fiala~\cite{beck1981integer} as a special case.

We write $\|\cdot\|_2$ and $\|\cdot\|_\infty$ for the Euclidean
and maximum norms, respectively.

\begin{conjecture}[Koml\'os conjecture]
There is a universal constant $C>0$ such that the following holds
for all positive integers $n,d$.
Let $v_1,\ldots,v_n \in \R^d$ satisfy $\|v_i\|_2 \le 1$.
Then there exist signs $\eps_1,\ldots,\eps_n \in \{-1,1\}$ such that
\[
\norm{\sum_{i=1}^n \eps_i v_i}_{\infty} \le C.
\]
\end{conjecture}

\paragraph{Progress toward the conjecture.}
Banaszczyk~\cite{banaszczyk1998balancing} established the bound
$O(\sqrt{\log n})$ through his vector-balancing theorem.
Subsequent algorithmic progress includes the Gram--Schmidt walk
of Bansal, Dadush, Garg, and Lovett~\cite{bansal2018gram},
which gives a constructive version of Banaszczyk's
vector-balancing theorem.
Bansal and Jiang~\cite{bansal2026decoupling,bansal2026exposition}
obtained the first asymptotic improvement over
Banaszczyk's bound, proving
$O((\log n)^{1/4}(\log\log n)^{7/4})$.
Ercan~\cite{ercan2026log} subsequently removed the iterated-logarithmic
factor, obtaining $O((\log n)^{1/4})$.

\paragraph{Recent breakthrough.}
In September 2026, Guo, Fang, and Lu~\cite{guo2026vector} gave a proof
of the Koml\'os conjecture with $C=3\sqrt{2\pi}$.
The authors credit the Odin Automatic AI Research Agent with
discovering the proof.
Their argument uses Banaszczyk's convex-body transform and
a variational analysis of the directional total variation
of probability densities.
Its starting density and translation-overlap estimates build on
Smirnov and Vershynin~\cite{smirnov2026fisher}, who relate Fisher
information to online discrepancy with discarded steps.
The present work grew out of an effort to understand their
argument and extract an elementary proof.

\paragraph{Our results.}
We give an elementary proof of the Koml\'os conjecture,
combining a combinatorial balancing argument with a continuous
construction followed by discretization. The proof does not
require advanced convex geometry or analysis.
Our proof yields the constant $36$, larger than
$3\sqrt{2\pi}$ in~\cite{guo2026vector}.
We have prioritized clarity of exposition and made no attempt
to optimize the constant.
We note that our proof gives a finite construction for rational inputs, but does not establish a polynomial-time algorithm.

\begin{theorem}
\label{thm:komlos}
Let $v_1,\ldots,v_n \in \R^d$ satisfy $\|v_i\|_2 \le 1$.
Then there exist signs $\eps_1,\ldots,\eps_n \in \{-1,1\}$ such that
\[
\norm{\sum_{i=1}^n \eps_i v_i}_{\infty} \le 36.
\]
\end{theorem}

\paragraph{Notation.}
For a nonnegative function $P$ on a domain $X$, write
$\supp(P):=\{x\in X:P(x)>0\}$ for its support.
In the discrete parts of the proof, all probability distributions
have finite support and total mass $\sum_xP(x)=1$.
We write $\Pr_P[A]:=\sum_{x\in A}P(x)$ for the probability of an
event $A$, and $\mu(P):=\sum_xP(x)x$ for the mean when $X=\R^d$.
For $S\subseteq\R^d$, $\operatorname{conv}(S)$ denotes its convex
hull. In particular, $\mu(P)\in\operatorname{conv}(\supp(P))$.

\paragraph{Proof approach.}
Our proof is based on distributions that are nearly invariant under
prescribed translations. The central quantity is their \emph{shift
distance}, which we now define.

For finitely supported probability distributions $P,Q$ on the same
domain, their statistical distance (or total variation distance) is
$\SD{P}{Q}:=\frac12\sum_x|P(x)-Q(x)|$.

\begin{definition}[Shift distance]
\label{def:shift-distance}
Let $P$ be a finitely supported probability distribution on $\R^d$,
and let $u\in\R^d$. Write $P+u$ for the distribution of $X+u$
when $X\sim P$; thus $(P+u)(x)=P(x-u)$.
The \emph{shift distance} of $P$ at $u$ is
\[
\SDS{P}{u}
:=\SD{P}{P+u}
=\frac12\sum_{x\in\R^d}|P(x)-P(x-u)|.
\]
\end{definition}

Thus, small shift distance means that $P$ changes little under the
specified translation.
Equivalently, since
\[
\sum_{x\in\R^d}\min\{P(x),P(x-u)\}=1-\SDS{P}{u},
\]
$\SDS{P}{u}\le\delta$ means that $P$ and its translate have at
least $1-\delta$ of their probability mass in common. This overlap
interpretation is what makes shift distance useful in the splitting
argument below.

The proof has two main parts. First, we show that the mean of a
distribution with small shift distances at $6v_1,\ldots,6v_n$ can
be shifted by a signed sum of the $v_i$ while remaining in the
convex hull of its support. We allow an arbitrary mean so that
this statement can be applied inductively after adding a coordinate.
Second, we construct a mean-zero distribution with the required
shift bounds in a cube by rounding a continuous product density
to a rational grid.

\begin{lemma}[From near invariance to signed sums]
\label{lemma:balancing}
Let $v_1,\ldots,v_n\in\R^d$. Suppose $P$ is a finitely supported
probability distribution on $\R^d$ with
\[
\SDS{P}{6v_i}\le\frac13
\qquad (i=1,\ldots,n).
\]
Then there are signs $\eps_1,\ldots,\eps_n\in\{-1,1\}$ such that
\[
\mu(P)+\sum_{i=1}^n\eps_i v_i\in\operatorname{conv}(\supp(P)).
\]
\end{lemma}

\begin{lemma}[A near-invariant distribution in the cube]
\label{lemma:cube}
Let $v_1,\ldots,v_n\in\Q^d$ satisfy $\|v_i\|_2\le1$.
There is a finitely supported distribution $P$ on $[-6,6]^d$
with mean zero and
\[
\SDS{P}{v_i}\le\frac13
\qquad (i=1,\ldots,n).
\]
\end{lemma}

\begin{proof}[Proof of \Cref{thm:komlos}]
First suppose the vectors are rational. Take $P$ from
\Cref{lemma:cube} and apply \Cref{lemma:balancing} to the vectors
$v_i/6$. Since $P$ has mean zero, this gives signs such that
\[
\frac16\sum_{i=1}^n\eps_i v_i
\in\operatorname{conv}(\supp(P))\subseteq[-6,6]^d,
\]
and hence $\|\sum_i\eps_i v_i\|_\infty\le36$.
For real vectors, approximate each $v_i$ by rational vectors in the
set $\{x\in\R^d:\|x\|_2\le1\}$, for instance by truncating its coordinates
toward zero. Since there are only finitely many sign vectors, some
fixed choice of signs works for an infinite subsequence of these approximations.
Passing to the limit proves the same bound.
\end{proof}

\paragraph{Prefix discrepancy.}
Our theorem controls the final signed sum.
The \emph{strong Koml\'os conjecture} (also called the prefix
Koml\'os conjecture) asked whether every sequence
$v_1,\ldots,v_n\in\R^d$ with $\|v_i\|_2\le1$ admits signs
$\eps_1,\ldots,\eps_n\in\{-1,1\}$ satisfying
\[
\max_{1\le k\le n}
\norm{\sum_{i=1}^k\eps_i v_i}_{\infty}\le C.
\]
Here $C$ is a universal constant, the order is fixed, and the
signs may depend on the entire sequence.
Prefix balancing goes back to Spencer~\cite{spencer1977balancing};
the constant bound was formulated explicitly by Bansal, Jiang,
Meka, Singla, and Sinha~\cite[Open Problem~6.1 in the full version]{bansal2022prefix}.
Guo, Fang, and Lu~\cite[Section~1.2]{guo2026vector} also note that
their proof does not control all prefixes.
Kintali~\cite{kintali2026strong} disproved the constant-bound
conjecture, constructing square $n\times n$ examples with prefix
discrepancy $\Omega(\sqrt{\log\log n})$.
The best known bound uniform in $d$ remains $O(\sqrt{\log n})$,
due to Banaszczyk~\cite{banaszczyk2012series}.\footnote{
To remove the dimension dependence, let $r_1,\ldots,r_d\in\R^n$
be the rows of the matrix with columns $v_1,\ldots,v_n$.
A row with $\|r_k\|_1\le1$ has discrepancy at most $1$ for every
signing and every prefix. Each remaining row satisfies
$\|r_k\|_2^2\ge\|r_k\|_1^2/n>1/n$.
Since
$\sum_{k=1}^d\|r_k\|_2^2=\sum_{i=1}^n\|v_i\|_2^2\le n$,
fewer than $n^2$ rows remain.
Restricting to these rows does not increase any column's Euclidean
norm. Applying the $O(\sqrt{\log d+\log n})$ prefix bound of Banaszczyk~\cite{banaszczyk2012series}
in this reduced dimension gives $O(\sqrt{\log n})$ for $n\ge2$;
see also~\cite[Theorem~1.1]{bansal2022prefix}.
If no rows remain, or if $n=1$, the discrepancy is at most $1$.
}
Aden-Ali~\cite{aden2026optimal} gives a randomized online algorithm
that, for every input sequence fixed independently of the algorithm's
random choices, chooses each sign as its vector arrives and achieves
$O(\sqrt{\log n})$ prefix discrepancy with high probability.
Determining the optimal growth rate of fixed-order prefix
discrepancy remains open.

\paragraph{Related work.}
In unpublished work, Jain~\cite{jain2026fisher} obtained the constant
$3\pi$ by proving preservation of a Fisher information matrix bound
under Banaszczyk's transform, building on Smirnov--Vershynin and
Guo--Fang--Lu.
Guillen and Kobzar~\cite{guillen2026complex} proved a constant bound
for complex discrepancy, where signs are replaced by complex
coefficients of modulus one, using the Bellman function method
and the framework of Bansal and Jiang.
For real input vectors, this is equivalent to rank-two vector
discrepancy and also yields a constant bound for Gaussian discrepancy.
Akbas and Sra~\cite{akbas2026boolean} combined the signing theorem of
Guo, Fang, and Lu with Gaussian matrix estimates and a replica
argument to prove Boolean matrix small-ball inequalities, obtaining
applications to the Kadison--Singer problem, the matrix Spencer
conjecture, and the Koml\'os conjecture.

\paragraph{Subsequent developments.}
After this paper appeared online, Guo, Fang, and Lu~\cite{guo2026polynomial} gave a
deterministic polynomial-time algorithm for constant-discrepancy
Koml\'os signing in the unit-cost real-RAM model.
Bandeira~\cite{bandeira2026simplification} gave another simplification
of their original proof, using Dirichlet energy and Banaszczyk's
transform to obtain the constant $9\pi$.
Dahia~\cite{dahia2026komlos} formalized our proof with constant $36$
in Lean~4.
We also learned of Kintali's counterexample to the fixed-order
strong Koml\'os conjecture~\cite{kintali2026strong}, discussed above.

\paragraph{AI methodology.}
We used ChatGPT-6 Astra to iteratively refine the proof and the
writing, with an emphasis on simplifying the argument and improving
the clarity of the exposition.

\paragraph{Acknowledgements.}
We thank Nikhil Bansal and Raghu Meka for helpful feedback and suggestions on an earlier draft.

\paragraph{Paper organization.}
We give an overview of both lemmas in \Cref{sec:overview}, then prove
\Cref{lemma:balancing} in \Cref{sec:balancing} and
\Cref{lemma:cube} in \Cref{sec:cube}.

\section{Proof overview}
\label{sec:overview}

\subsection{From near invariance to signed sums}

We prove \Cref{lemma:balancing} by induction on the number of
vectors, simultaneously in all dimensions. At each step, we split
once in the direction of the last vector, apply induction to the
remaining vectors, and pull the resulting convex combination back
to the original support.

\paragraph{One split.}
Let $v=v_n$. We split $P$ to form a distribution $Q$ on
$\R^d\times\{0,1\}$. A new state $(x,b)$ has possible parents
$x+3v$ and $x-3v$ in $P$; a parent is available if it has positive
mass in $P$. At each position $x$, the split assigns half the
larger parent mass to $(x,0)$ and half the smaller parent mass to
$(x,1)$. Every state of positive mass therefore has an available
parent, and a state with $b=1$ has both.

We regard the bit $b$ as an extra real coordinate. The first $d$
coordinates retain their mean $\mu(P)$. The last coordinate has
mean $\beta:=\Pr_Q[b=1]$, which is half the overlap of the two
translates of $P$. Since these translates differ by $6v$, the
assumption $\SDS{P}{6v}\le1/3$ gives $\beta\ge1/3$.
Thus $\mu(Q)=(\mu(P),\beta)$.

\paragraph{Induction.}
Extend each remaining vector to $(v_i,0)\in\R^{d+1}$.
Splitting does not increase shift distance in these directions,
so induction applies to $Q$. It gives signs for the first $n-1$
vectors and a point
\[
(\mu(P)+s,\beta)\in\operatorname{conv}(\supp(Q)),
\qquad
s:=\sum_{i=1}^{n-1}\eps_i v_i.
\]
Represent this point as the mean of a probability distribution $R$
supported on $\supp(Q)$. Its last coordinate forces
$\Pr_R[b=1]=\beta\ge1/3$. This is why we keep the extra
coordinate when applying induction: it ensures that the chosen
convex combination still has enough mass with two available parents.

\paragraph{Pullback.}
Send each entry of $R$ with $b=0$ to any available parent in $P$.
These entries have total mass at most two-thirds, and each move
is by $\pm3v$. Their contribution to the change in the first $d$
coordinates of the mean is therefore $av$ for some $|a|\le2$.
The entries with $b=1$ have both parents available and total mass
at least one-third. By dividing this mass between its parents,
we can contribute any scalar multiple of $v$ with coefficient
between $-1$ and $1$.
Choose $\eps_n\in\{-1,1\}$ with $|\eps_n-a|\le1$ and use
this mass to contribute $(\eps_n-a)v$. The resulting distribution
is supported on $\supp(P)$ and has mean
$\mu(P)+s+\eps_n v$, completing the induction step.

\paragraph{Connection to the original proof.}
Guo, Fang, and Lu~\cite[Section~4]{guo2026vector} introduce an
extra coordinate and rearrange a density along it, placing larger
values closer to zero. This preserves mass and does not increase
total variation distance between translates in the original
coordinates. Our splitting operator uses a bit as the extra
coordinate and sorts two translated probabilities. Preserving
the mean of this coordinate under induction supplies the mass
needed for the pullback. This simplifies the argument, at the
cost of a larger constant in the final bound.

\subsection{A near-invariant distribution in the cube}

For \Cref{lemma:cube}, we first construct a continuous probability
density $F$ on $[-6,6]^d$. We take $F$ to be a product of
one-dimensional densities, each proportional to $\max\{6-|t|,0\}^2$.
To estimate the effect of translation, we work with $f=\sqrt{F}$,
a product of tent functions. When we square the derivative of $f$
in direction $v$ and integrate, the mixed terms vanish by symmetry,
leaving a multiple of $\sum_k v_k^2=\|v\|_2^2$. The fundamental
theorem of calculus and Cauchy--Schwarz then give total variation
distance at most $\|v\|_2/\sqrt{12}$ for a translation by $v$.

We then discretize the construction to make it finitely supported. We choose a positive integer $N$ such that the grid
$N^{-1}\mathbb{Z}^d$ contains all the given rational vectors, and
round each coordinate to the nearest grid point. Since $\pm6$
are grid points, rounding keeps the support in the cube. It also
preserves symmetry, giving a finitely supported distribution with
mean zero. Rounding commutes with translations by grid vectors
and cannot increase total variation
distance. The same bound therefore holds for the rounded
distribution. Since $\|v_i\|_2\le1$, this bound is less than
$1/3$ for each given vector.

\section{From near invariance to signed sums}
\label{sec:balancing}

The proof of \Cref{lemma:balancing} uses one split, an application
of induction, and one pullback. The split adds a coordinate taking
values in $\{0,1\}$, so we view its output as a distribution on
$\R^{d+1}$. For $u\in\R^d$, write $\bar u:=(u,0)$ for the
corresponding vector in $\R^{d+1}$.

\begin{definition}[Splitting operator]
\label{def:splitting}
Let $v\in\R^d$, and let $P$ be a finitely supported nonnegative
function on $\R^d$. Define $T_vP$ on $\R^{d+1}$ by
\[
\begin{aligned}
(T_vP)(x,0)&:=\frac12\max\{P(x+v),P(x-v)\},\\
(T_vP)(x,1)&:=\frac12\min\{P(x+v),P(x-v)\},
\end{aligned}
\]
and set it to zero whenever the last coordinate is not $0$ or $1$.
\end{definition}

We call $x+v$ and $x-v$ the two possible parents of $(x,b)$.
A parent is available if it belongs to $\supp(P)$. Every state
in $\supp(T_vP)$ has at least one available parent, and a state
with $b=1$ has both. We will use this freedom in the pullback.

\paragraph{Mass and mean.}
Adding the maximum and minimum gives
\begin{equation}
\label{eq:split-sum}
\sum_{b=0}^1(T_vP)(x,b)
=\frac12\bigl(P(x+v)+P(x-v)\bigr).
\end{equation}
Summing over $x$ shows that $T_v$ preserves total mass, even for
unnormalized nonnegative inputs. In particular, it maps probability
distributions to probability distributions.

Now suppose $P$ is a probability distribution. The two opposite
shifts cancel in the mean of the first $d$ coordinates:
\[
\sum_{x,b}(T_vP)(x,b)x
=\frac12\sum_x(x-v)P(x)+\frac12\sum_x(x+v)P(x)
=\mu(P).
\]
The mean of the last coordinate is its probability of being $1$.
By the overlap interpretation of shift distance,
\begin{equation}
\label{eq:one-bit}
\begin{aligned}
\Pr_{(x,b)\sim T_vP}[b=1]
&=\frac12\sum_x\min\{P(x+v),P(x-v)\}\\
&=\frac12\sum_x\min\{P(x),P(x-2v)\}\\
&=\frac{1-\SDS{P}{2v}}2.
\end{aligned}
\end{equation}
Thus, writing $\beta:=\Pr_{T_vP}[b=1]$, we have
\[
\mu(T_vP)=(\mu(P),\beta).
\]

\begin{claim}[Shift distance does not increase]
\label{claim:shift-contraction}
Let $P$ be a finitely supported probability distribution on
$\R^d$, and let $u,v\in\R^d$. Then
\begin{equation}
\label{eq:shift-contraction}
\SDS{T_vP}{\bar u}\le\SDS{P}{u}.
\end{equation}
\end{claim}

\begin{proof}
Both maximum and minimum are nondecreasing in each argument,
so $T_v$ preserves pointwise inequalities between nonnegative
functions. The definition also gives
$T_v(P+u)=T_vP+\bar u$.

Consider the mass common to $P$ and $P+u$,
given by $H(x):=\min\{P(x),P(x-u)\}$.
Since $H\le P$ and $H\le P+u$ pointwise, applying $T_v$ to
these inequalities gives
\[
T_vH\le\min\{T_vP,T_v(P+u)\}
=\min\{T_vP,T_vP+\bar u\}.
\]
Since $T_v$ preserves the total mass of $H$,
\[
\begin{aligned}
1-\SDS{T_vP}{\bar u}
&=\sum_{x,b}\min\{(T_vP)(x,b),(T_vP+\bar u)(x,b)\}\\
&\ge\sum_{x,b}(T_vH)(x,b)\\
&=\sum_xH(x)
=1-\SDS{P}{u}.
\end{aligned}
\]
This proves \eqref{eq:shift-contraction}.
\end{proof}

\begin{proof}[Proof of \Cref{lemma:balancing}]
We induct on $n$, proving the statement simultaneously for all
dimensions $d$. For $n=0$, the conclusion is
$\mu(P)\in\operatorname{conv}(\supp(P))$, which holds because
$\mu(P)$ is a convex combination of the support points.

Suppose $n\ge1$, let $v=v_n$, and set $Q:=T_{3v}P$.
By \eqref{eq:one-bit},
\[
\beta:=\Pr_{(x,b)\sim Q}[b=1]
=\frac{1-\SDS{P}{6v}}2\ge\frac13,
\qquad
\mu(Q)=(\mu(P),\beta).
\]
For $i<n$, \Cref{claim:shift-contraction} gives
$\SDS{Q}{6\bar v_i}\le\SDS{P}{6v_i}\le1/3$.
Apply the induction hypothesis to $Q$ and
$\bar v_1,\ldots,\bar v_{n-1}\in\R^{d+1}$.
It yields signs $\eps_1,\ldots,\eps_{n-1}$ such that
\[
(\mu(P)+s,\beta)\in\operatorname{conv}(\supp(Q)),
\qquad
s:=\sum_{i=1}^{n-1}\eps_i v_i.
\]
Choose a probability distribution $R$ supported on $\supp(Q)$
whose mean is this point. Since the last coordinate takes only
values $0$ and $1$, its mean determines its probability of being $1$:
\[
\Pr_{(x,b)\sim R}[b=1]=\beta\ge\frac13.
\]

We now pull $R$ back to $\supp(P)$ by moving each state $(x,b)$
to its available parents $x+3v$ and $x-3v$. For the states with
$b=0$, choose any available parent and send the entire mass there.
Their total mass is $1-\beta$, so their contribution to the change
in the first $d$ coordinates of the mean is $av$ for some scalar
$a$ satisfying
\[
|a|\le3(1-\beta)\le2.
\]

Every state with $b=1$ has both parents available. Send a common
fraction $t\in[0,1]$ of each such state's mass to $x+3v$ and
the remaining fraction $1-t$ to $x-3v$. Since these states have
total mass $\beta$, their contribution to the change in the mean is
\[
3\beta t v-3\beta(1-t)v=3\beta(2t-1)v.
\]
As $t$ varies from $0$ to $1$, the coefficient ranges over
$[-3\beta,3\beta]$.
Choose $\eps_n\in\{-1,1\}$ to have the sign of $a$, taking
$\eps_n=1$ when $a=0$. Since $a\in[-2,2]$, we have
\[
|\eps_n-a|\le1\le3\beta.
\]
Set
\[
t:=\frac12+\frac{\eps_n-a}{6\beta}.
\]
The preceding inequality ensures that $t\in[0,1]$, and
$3\beta(2t-1)=\eps_n-a$. Thus the one-bit states contribute
exactly $(\eps_n-a)v$. Together with the contribution $av$ from the
zero-bit states, the total change in the first $d$ coordinates
of the mean is exactly $\eps_n v$.

Let $S$ be the resulting distribution on $\R^d$, adding
contributions that reach the same parent. Every state transfers
its entire mass to available parents, so $S$ is a
probability distribution supported on $\supp(P)$. Its mean is
\[
\mu(S)
=\mu(P)+s+\eps_n v
=\mu(P)+\sum_{i=1}^n\eps_i v_i
\in\operatorname{conv}(\supp(P)),
\]
as required.
\end{proof}

\section{A near-invariant distribution in the cube}
\label{sec:cube}

We first prove a continuous version of \Cref{lemma:cube}, using
a product density whose square root has small directional
derivatives. We then deduce \Cref{lemma:cube} by rounding this
density to a grid.

A probability density is a nonnegative integrable function
$F:\R^d\to\R$ with $\int_{\R^d}F(x)\,dx=1$.
We say that $F$ is supported on $S$ if it vanishes outside $S$.
For two densities, statistical distance is defined by replacing
the sum with an integral:
$\SD{F}{H}:=\frac12\int_{\R^d}|F(x)-H(x)|\,dx$.
We write $\|h\|_{L^2}:=(\int_{\R^d}|h(x)|^2\,dx)^{1/2}$ for the
$L^2$ norm, $\nabla h$ for the gradient, and $v\cdot\nabla h$ for
the directional derivative along $v$.

\begin{lemma}[A continuous near-invariant distribution in the cube]
\label{lemma:continuous-cube}
For every positive integer $d$, there is a continuous probability
density $F$ supported on $[-6,6]^d$, symmetric about the origin,
such that for every $v\in\R^d$,
\[
\frac12\int_{\R^d}|F(x+v)-F(x)|\,dx
\le\frac{\|v\|_2}{\sqrt{12}}.
\]
\end{lemma}

\begin{proof}
Define
\[
b(t):=\frac1{12}\max\{6-|t|,0\},
\qquad
f(x):=\prod_{k=1}^d b(x_k),
\qquad
F(x):=f(x)^2.
\]
The elementary identities
\[
\int_{\R}b^2=1,
\qquad
\int_{\R}(b')^2=\frac1{12},
\qquad
\int_{\R}bb'=0
\]
show that $F$ is a continuous probability density supported on
$[-6,6]^d$. Since $b$ is even, $F(-x)=F(x)$, so the density is
symmetric about the origin and has mean zero.
These identities also give, for every $v\in\R^d$,
\begin{equation}
\label{eq:directional-energy}
\int_{\R^d}(v\cdot\nabla f(x))^2\,dx
=\frac{\|v\|_2^2}{12}.
\end{equation}
Indeed, the mixed terms vanish because they contain the factor
$\int bb'=0$, while the $k$th diagonal term is $v_k^2/12$.

Because $b$ is continuous and linear on each of finitely many
intervals, the fundamental theorem of calculus along line segments
gives, for almost every $x$,
\[
f(x+v)-f(x)=\int_0^1 v\cdot\nabla f(x+tv)\,dt.
\]
Values of the derivatives at the breakpoints do not affect any
of the integrals. Cauchy--Schwarz therefore gives
\[
\begin{aligned}
\int_{\R^d}|f(x+v)-f(x)|^2\,dx
&\le\int_0^1\int_{\R^d}
   |v\cdot\nabla f(x+tv)|^2\,dx\,dt\\
&=\frac{\|v\|_2^2}{12}.
\end{aligned}
\]
Factoring a difference of squares and applying Cauchy--Schwarz
once more, we obtain
\begin{equation}
\label{eq:continuous-shift}
\begin{aligned}
\frac12\int_{\R^d}|F(x+v)-F(x)|\,dx
&\le\frac12\|f(\cdot+v)-f\|_{L^2}
                  \|f(\cdot+v)+f\|_{L^2}\\
&\le\|f(\cdot+v)-f\|_{L^2}
\le\frac{\|v\|_2}{\sqrt{12}},
\end{aligned}
\end{equation}
where we used $\|f\|_{L^2}=\|f(\cdot+v)\|_{L^2}=1$.
\end{proof}

We now round the continuous density to obtain the finitely
supported distribution required by \Cref{lemma:cube}.

\begin{proof}[Proof of \Cref{lemma:cube}]
Let $F$ be the symmetric density from \Cref{lemma:continuous-cube}.
Choose a positive integer $N$ such that $Nv_i\in\mathbb{Z}^d$
for every $i$, and let
$G=N^{-1}\mathbb{Z}^d=\{k/N:k\in\mathbb{Z}^d\}$ be the resulting grid.
Round each coordinate of a sample from $F$ to the nearest multiple
of $1/N$. The resulting distribution is
\[
P(x):=\int_C F(x+z)\,dz
\quad (x\in G),
\qquad
C:=\left[-\frac1{2N},\frac1{2N}\right)^d,
\]
and $P(x)=0$ outside $G$. Because $-6$ and $6$ are grid points,
rounding keeps the support in $[-6,6]^d$, so $P$ has finite
support. Symmetry gives $P(-x)=P(x)$ and hence $\mu(P)=0$;
cell boundaries have zero probability.

Rounding commutes with grid translations and cannot increase
total variation. Explicitly, for $v\in G$,
\[
\begin{aligned}
\SDS{P}{v}
&=\frac12\sum_{x\in G}
\left|\int_C\bigl(F(x+z-v)-F(x+z)\bigr)\,dz\right|\\
&\le\frac12\int_{\R^d}|F(z-v)-F(z)|\,dz
\le\frac{\|v\|_2}{\sqrt{12}}.
\end{aligned}
\]
The first inequality uses that the cells $x+C$, $x\in G$, partition
$\R^d$; the last uses \Cref{lemma:continuous-cube} with $-v$.
Applying this to each $v_i$ yields
$\SDS{P}{v_i}\le1/\sqrt{12}<1/3$, as required.
\end{proof}

\bibliographystyle{abbrv}
\bibliography{komlos}

\end{document}